%% file: SemiInfSE-Sinc.tex
\documentclass[preprint,3p,12pt,times]{elsarticle}

\journal{Elsevier}

\usepackage{mathrsfs}
\usepackage{amsmath}

\DeclareMathOperator{\Order}{O}
\DeclareMathOperator{\rme}{e}
\DeclareMathOperator{\imnum}{i}
\DeclareMathOperator{\arcsinh}{arcsinh}
\renewcommand{\Re}{\operatorname{Re}}
\renewcommand{\Im}{\operatorname{Im}}
\newdefinition{defn}{Definition}[section]
\newdefinition{exam}{Example}
\newtheorem{thm}{Theorem}[section]
\newtheorem{lem}[thm]{Lemma}
\newproof{proof}{Proof}
\newcommand{\domD}{\mathscr{D}}
\newcommand{\Hone}{\mathbf{H}^1}
\newcommand{\divv}{\mathrm{d}}
\newcommand{\diff}{\,\divv}
\renewcommand{\pi}{\piup}
\numberwithin{equation}{section}

\begin{document}

\begin{frontmatter}

\title{New conformal map for the Sinc approximation
for exponentially decaying functions
over the semi-infinite interval\tnoteref{mytitlenote}}
\tnotetext[mytitlenote]{This work was partially supported by JSPS Grant-in-Aid for Young Scientists (B) JP17K14147.}

\author{Tomoaki Okayama}
\address{Graduate School of Information Sciences, Hiroshima City University,
3-4-1, Ozuka-higashi, Asaminami-ku, Hiroshima 731-3194, Japan}
\ead{okayama@hiroshima-cu.ac.jp}

\author{Yuya Shintaku}
\address{Kanon Junior High School, 3-4-6, Minami-kanon, Nishi-ku, Hiroshima 733-0035, Japan}

\author{Eisuke Katsuura}
\address{Higashihiroshima City Official Government, 8-29, Saijo-sakaemachi, Higashihiroshima 739-8601, Japan}




\begin{abstract}
The Sinc approximation has shown high efficiency for
numerical methods in many fields.
Conformal maps play an important role in the success, i.e.,
appropriate conformal map must be employed
to elicit high performance of the Sinc approximation.
Appropriate conformal maps have been proposed for typical cases; however,
such maps may not be optimal.
Thus, the performance of the Sinc approximation may be improved by using
another conformal map rather than an existing map.
In this paper, we propose a new conformal map
for the case where functions are defined over the semi-infinite interval
and decay exponentially.
Then, we demonstrate in both theoretical and numerical ways that
the convergence rate is improved by replacing the existing conformal map
with the proposed map.
\end{abstract}


\end{frontmatter}


\input{introduction.tex}

\input{error-bounds.tex}

\input{numer-exam.tex}

\input{proofs.tex}

\input{conclusion.tex}


\bibliography{SemiInfSE-Sinc}

\end{document}

%% file: introduction.tex
\section{Introduction}
The Sinc approximation is a highly efficient approximation formula
for analytic functions (described precisely later), expressed as
\begin{equation}
 F(x)\approx \sum_{k=-M}^{N} F(kh) S(k,h)(x),\quad x\in(-\infty,\infty),
\label{eq:Sinc-approx}
\end{equation}
where $S(k,h)(x)$ is the so-called Sinc function, defined as
\[
 S(k,h)(x)=\frac{\sin[\pi(x/h - k)]}{\pi(x/h - k)},
\]
and $M$, $N$, $h$ are selected according to
the given positive integer $n$.
This approximation gives root-exponential convergence,
if $|F(x)|$ decays exponentially as $x\to\pm\infty$.
Here, we should also note that the target interval to be considered is
the infinite interval $(-\infty,\infty)$.
Accordingly, $F$ should be defined over the infinite interval.
In the case where the function to be approximated decays exponentially
but is defined over the semi-infinite interval $(0,\infty)$,
e.g., $f(t)=\sqrt{t}\rme^{-t}$,
Stenger~\cite{stenger93:_numer_method} proposed the use of
a conformal map
\begin{equation}
 t=\psi(x)=\arcsinh(\rme^{x}),
 \label{eq:Stenger-map}
\end{equation}
whereby the transformed function $f(\psi(x))$
is defined over $(-\infty,\infty)$ and decays exponentially
as $x\to\pm\infty$. Therefore, we can apply the Sinc approximation
to $f(\psi(x))$ as
\[
f(\psi(x)) \approx \sum_{k=-M}^N f(\psi(kh)) S(k,h)(x),
\]
or equivalently,
\begin{equation}
f(t) \approx \sum_{k=-M}^N f(\psi(kh)) S(k,h)(\psi^{-1}(t)).
 \label{eq:Stenger-approx}
\end{equation}
In other cases, he also considered appropriate conformal maps.
As a result, numerical methods based on the Sinc approximation
demonstrate root-exponential convergence in many
fields~\cite{lund92:_sinc,stenger00:_summar,sugihara04:_recen,stenger11:_handb},
and such methods surpass conventional methods that converge polynomially.

The main objective of this paper is to improve
the conformal map~\eqref{eq:Stenger-map}.
A conformal map that maps $(-\infty,\infty)$ onto $(0,\infty)$
is not unique.
In addition, the convergence rate may be improved if we use another conformal map that performs the same role.
In fact, in the area of numerical integration,
convergence rate improvement has been reported~\cite{Machida,Hara}
by replacing the conformal map $t=\psi(x)$ with
\[
 t=\phi(x)=\log(1+\rme^x).
\]
Considering the above as motivation, this study proposes to combine the
Sinc approximation with $t=\phi(x)$ rather than $t=\psi(x)$ as
\begin{equation}
f(t) \approx \sum_{k=-M}^N f(\phi(kh)) S(k,h)(\phi^{-1}(t)).
 \label{eq:Shintaku-approx}
\end{equation}
Based on theoretical and numerical investigations,
we demonstrate that the convergence rate of
the approximation~\eqref{eq:Shintaku-approx} is better than that of
the approximation~\eqref{eq:Stenger-approx}.

The remainder of this paper is organized as follows.
In Section~\ref{sec:error-bounds},
we describe the error bound for the
the approximation~\eqref{eq:Stenger-approx} (existing result) and
the error bound for the approximation~\eqref{eq:Shintaku-approx}
(new result by this paper).
Furthermore, the difference between the two approximations is also discussed
in this section.
Numerical examples that support the theoretical results
are given in Section~\ref{sec:numer-exam}.
Proofs of the new result presented in Section~\ref{sec:error-bounds}
are given in Section~\ref{sec:proofs}, and conclusions and suggestions for future work are given in Section~\ref{sec:conclusion}.

%% file: error-bounds.tex
\section{Error bounds for the existing and new approximations}
\label{sec:error-bounds}

\subsection{Simply connected complex domain to be considered}

For the approximation~\eqref{eq:Sinc-approx}
to be accurate, $F$ should be analytic in
a strip domain
$\domD_d=\{\zeta\in\mathbb{C}:|\Im \zeta|< d\}$
for $d>0$.
Therefore, in the case of the approximation~\eqref{eq:Stenger-approx},
$f(\psi(\cdot))$ should be analytic in $\domD_d$.
This means that $f$ should be analytic in
\[
\psi(\domD_d) = \left\{
z\in\mathbb{C}: |\arg(\sinh z)| < d
\right\},
\]
which is a translated domain from $\domD_d$ by the conformal map $\psi$.
Similarly, in the case of the approximation~\eqref{eq:Shintaku-approx},
$f$ should be analytic in
\[
 \phi(\domD_d) = \left\{
z\in\mathbb{C}: |\arg(\rme^{z}-1)| < d
\right\},
\]
which is a translated domain from $\domD_d$ by the conformal map $\phi$.
Those two domains with $d=1$ are shown in Figures~\ref{Fig:Stenger-dom}
and~\ref{Fig:Shintaku-dom}.
Note that both domains include the semi-infinite interval $(0,\infty)$.

\begin{figure}[htbp]
\begin{center}
\begin{minipage}{0.49\linewidth}
{\centering
  \includegraphics[width=\linewidth]{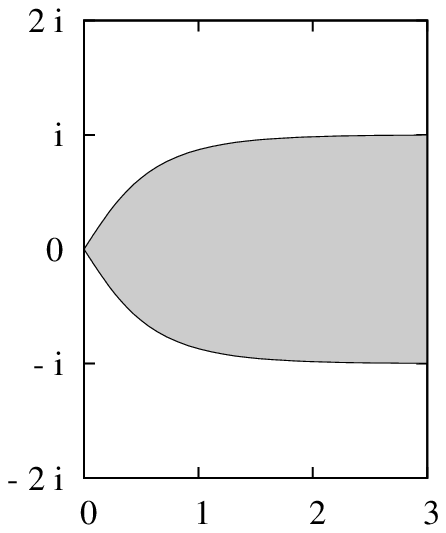}
  \caption{Simply connected complex domain $\psi(\domD_1)$.}
  \label{Fig:Stenger-dom}
}
\end{minipage}
\begin{minipage}{0.49\linewidth}
{\centering
  \includegraphics[width=\linewidth]{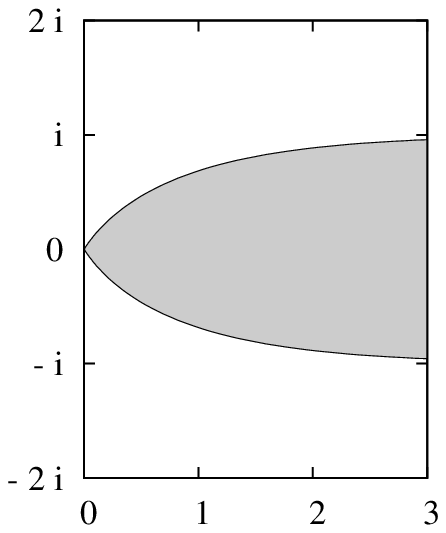}
  \caption{Simply connected complex domain $\phi(\domD_1)$.}
  \label{Fig:Shintaku-dom}
}
\end{minipage}
\end{center}
\end{figure}

\subsection{Theoretical results}

The error of the existing approximation~\eqref{eq:Stenger-approx}
was estimated as follows.

\begin{thm}[Okayama~{\cite[Theorem~2.4]{okayama16:_error}}]
\label{thm:existing}
Assume that $f$ is analytic in $\psi(\domD_d)$ with $0<d\leq \pi/2$
and there exist positive constants $K$, $\alpha$, and $\beta$ such that
\begin{equation}
|f(z)|\leq K \left|\frac{z}{1+z}\right|^{\alpha} |\rme^{-z}|^{\beta}
\label{eq:f-bound}
\end{equation}
holds for all $z\in\psi(\domD_d)$.
Let $\mu=\min\{\alpha,\beta\}$,
let $M$ and $N$ be defined as
\begin{equation}
\begin{cases}
M=n,\quad N=\lceil\alpha n/\beta\rceil
 & \,\,\,(\text{if}\,\,\,\mu = \alpha),\\
N=n,\quad M=\lceil\beta n/\alpha\rceil
 &  \,\,\,(\text{if}\,\,\,\mu = \beta),
\end{cases}
\label{eq:def-MN}
\end{equation}
and let $h$ be defined as
\begin{equation}
h = \sqrt{\frac{\pi d}{\mu n}}.
\label{eq:def-h}
\end{equation}
Then,
\begin{equation}
 \sup_{t\in(0,\infty)}\left|
  f(t) - \sum_{k=-M}^N f(\psi(kh))S(k,h)(\psi^{-1}(t))
 \right|
\leq C \sqrt{n} \rme^{-\sqrt{\pi d \mu n}}
\label{leq:existing-estimate}
\end{equation}
holds, where $C$ is a constant given by
\[
C = \frac{2 K}{\sqrt{\pi d \mu}}
\left\{
 \frac{2}{\sqrt{\pi d \mu}(1-\rme^{-2\sqrt{\pi d \mu}})\{\cos(d/2)\}^{\alpha+\beta}}2^{(\alpha+\beta)/2} + 1
\right\}.
\]
\end{thm}

In the same manner,
this study estimates the error of the new
approximation~\eqref{eq:Shintaku-approx}
as follows. The proof is given in Section~\ref{sec:proofs}.

\begin{thm}
\label{thm:new}
Assume that $f$ is analytic in $\phi(\domD_d)$ with $0<d<\pi$
and there exist positive constants $K$, $\alpha$, and $\beta$
such that~\eqref{eq:f-bound} holds for all $z\in\phi(\domD_d)$.
Let $\mu=\min\{\alpha,\beta\}$, let $M$ and $N$ be defined as~\eqref{eq:def-MN},
and let $h$ be defined as~\eqref{eq:def-h}.
Then,
\begin{equation}
 \sup_{t\in(0,\infty)}\left|
  f(t) - \sum_{k=-M}^N f(\phi(kh))S(k,h)(\phi^{-1}(t))
 \right|
\leq C \sqrt{n} \rme^{-\sqrt{\pi d \mu n}}
\label{leq:new-estimate}
\end{equation}
holds, where $C$ is a constant given by
\[
C = \frac{2 K}{\sqrt{\pi d \mu}}
\left\{
 \frac{2}{\sqrt{\pi d \mu}(1-\rme^{-2\sqrt{\pi d \mu}})\{\cos(d/2)\}^{\alpha+\beta}}\left(\frac{\rme}{\rme - 1}\right)^{\mu/2} + 1
\right\}.
\]
\end{thm}

\subsection{Discussion}

In view of~\eqref{leq:existing-estimate}
and~\eqref{leq:new-estimate},
their convergence rates appear to be the same, i.e., $\Order(\sqrt{n}\rme^{-\sqrt{\pi d \mu n}})$.
However, there is a difference in the condition of $d$,
i.e., $0<d\leq\pi/2$ in Theorem~\ref{thm:existing}
and $0<d<\pi$ in Theorem~\ref{thm:new}.
This means that $d$ of the new approximation may be greater than
$d$ of the existing approximation.
In this case, the new approximation~\eqref{eq:Shintaku-approx}
converges faster than~\eqref{eq:Stenger-approx}.

This difference in the range of $d$ originates from the conformal maps
$\psi$ and $\phi$. By observing the derivatives of the functions
\begin{align*}
 \psi'(\zeta)=\frac{1}{\sqrt{1+\rme^{-2\zeta}}},
\quad
 \phi'(\zeta)=\frac{1}{1+\rme^{-\zeta}},
\end{align*}
we see that $\psi(\zeta)$ is not analytic at $\zeta=\pm\imnum (\pi/2)$,
and $\phi(\zeta)$ is not analytic at $\zeta=\pm \imnum \pi$.
Accordingly, $f(\psi(\zeta))$ is analytic at most $\domD_{\pi/2}$,
and $f(\phi(\zeta))$ is analytic at most $\domD_{\pi}$.
Therefore, the range of $d$ is $0<d\leq \pi/2$ in Theorem~\ref{thm:existing}
and is $0<d<\pi$ in Theorem~\ref{thm:new}.
Note that we cannot permit $d=\pi$ in Theorem~\ref{thm:new}
because the denominator of $C$ includes $\cos(d/2)$.

%% file: numer-exam.tex
\section{Numerical examples} \label{sec:numer-exam}

Numerical results are presented in this section.
All computation programs were written in C with
double-precision floating-point arithmetic.
The programs and computation results are available online at
\texttt{https://github.com/okayamat/sinc-expdecay-semiinf}.

We consider the following examples.

\begin{exam}[{\cite[Example 3]{okayama16:_error}}]
\label{Exam:f_1}
Consider the function
\[
 f_1(t)=t^{\pi/4}\rme^{-t},
\]
which satisfies the assumptions
in Theorem~\ref{thm:existing}
with $\alpha=\pi/4$, $\beta=1 - \alpha/\pi$, $d=\pi/2$,
and $K=(1+(\pi/2)^2)^{\alpha/2}$,
and also
those in Theorem~\ref{thm:new}
with $\alpha=\pi/4$, $\beta=1 - \alpha/(2\pi)$,
$d=3$, and $K=[\{(1 - \gamma)^2 + \pi^2\}\rme^{\gamma/\pi}]^{\alpha/2}$,
where $\gamma=-\log(\cos(d/2))$.
\end{exam}
\begin{exam}
\label{Exam:f_2}
Consider the function
\[
 f_2(t)=\sqrt{\rme^{t} - 1}\rme^{-3t/2},
\]
which satisfies the assumptions
in Theorem~\ref{thm:existing}
with $\alpha=1/2$, $\beta=1$, $d=\pi/2$, and $K=4^{\alpha}$,
and also
those in Theorem~\ref{thm:new}
with $\alpha=1/2$, $\beta=1$, $d=3$, and
$K=\{\gamma (1 + \log(1+\gamma))/\log(1+\gamma)\}^{\alpha}$,
where $\gamma=1 + 1/\cos(d/2)$.
\end{exam}
\begin{exam}
\label{Exam:f_3}
Consider the function
\[
 f_3(t) = \sqrt{1 + (1 - 2\rme^{-t})^2}\frac{t}{1+t}\rme^{-t},
\]
which satisfies the assumptions
in Theorem~\ref{thm:existing}
with $\alpha=\beta=1$, $d=\arctan(3)$, and $K=\sqrt{2}$,
and also
those in Theorem~\ref{thm:new}
with $\alpha=\beta=1$, $d=\pi/2$, and
$K=2$.
\end{exam}

Numerical results for the three functions are shown in
Figures~\ref{Fig:example1}--\ref{Fig:example3}, where
``Observed error'' denotes the maximum value
of the absolute error investigated at the following 201 points:
$t = 2^{-50}$, $2^{-49.5}$, $\ldots,$
$2^{-0.5}$, $2^{0}$, $2^{0.5}$, $\ldots,$ $2^{50}$.
As can be seen in each graph,
the new approximation~\eqref{eq:Shintaku-approx}
converges faster than~\eqref{eq:Stenger-approx}.
Furthermore,
the error bound by Theorems~\ref{thm:existing}
and~\ref{thm:new} (dotted line)
clearly bounds the observed error (solid line).

\begin{figure}[htbp]
{\centering
  \includegraphics[scale=0.95]{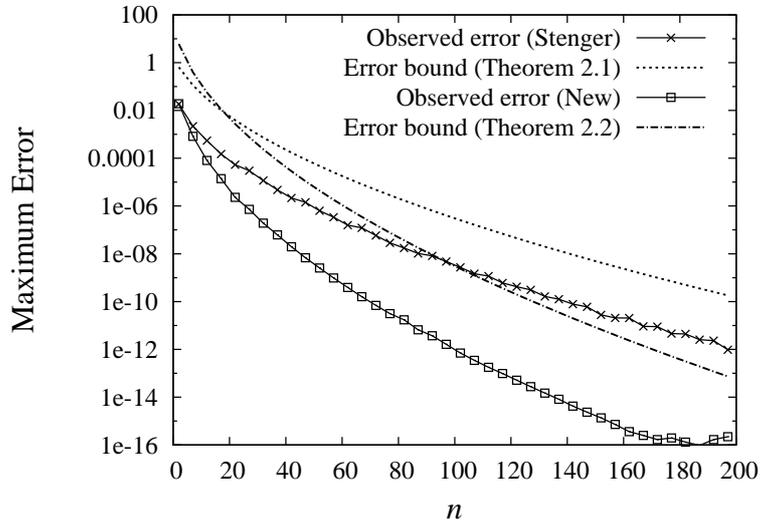}
  \caption{Approximation errors of $f_1$ and their bounds.}
  \label{Fig:example1}
}
\end{figure}
\begin{figure}[htbp]
{\centering
  \includegraphics[scale=0.95]{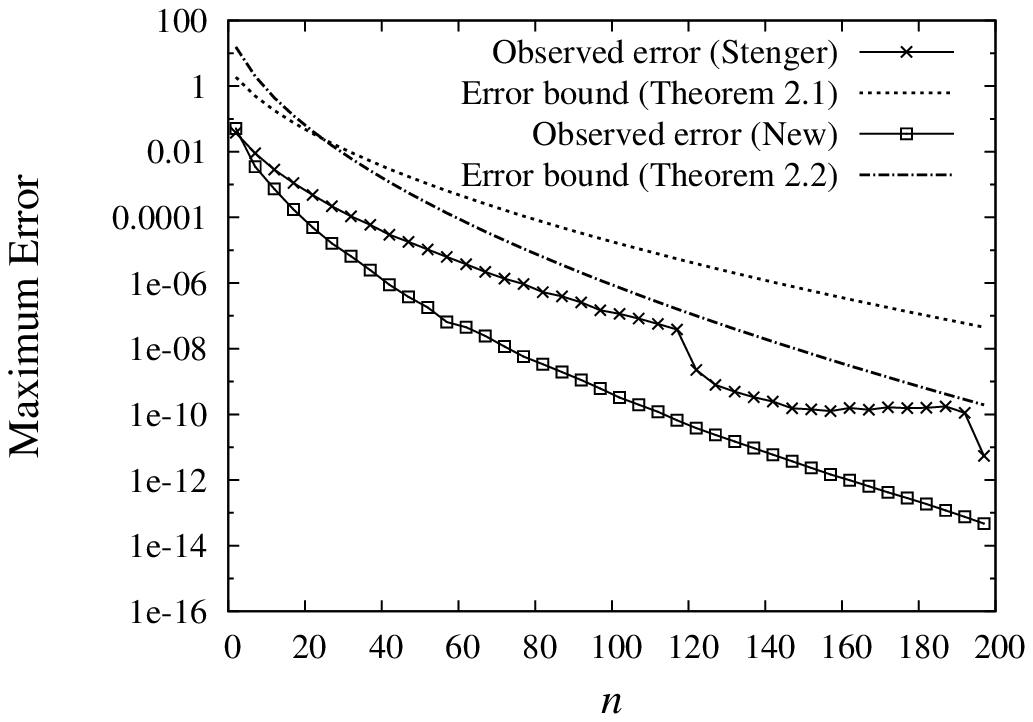}
  \caption{Approximation errors of $f_2$ and their bounds.}
  \label{Fig:example2}
}
\end{figure}
\begin{figure}[htbp]
{\centering
  \includegraphics[scale=0.95]{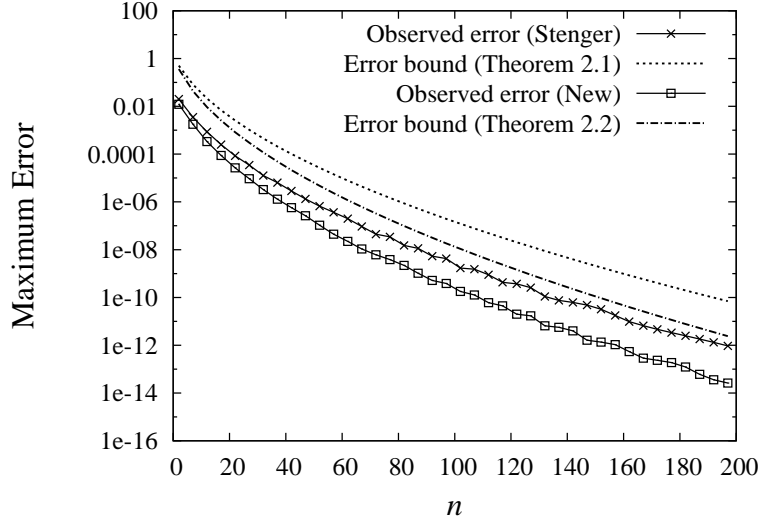}
  \caption{Approximation errors of $f_3$ and their bounds.}
  \label{Fig:example3}
}
\end{figure}

%% file: proofs.tex
\section{Proofs}
\label{sec:proofs}

In this section, we prove Theorem~\ref{thm:new}.

\subsection{Proof sketch}

First, by applying $t=\phi(x)$ and putting $F(x)=f(\phi(x))$, we obtain
\[
\left|f(t) - \sum_{k=-M}^N f(\phi(kh))S(k,h)(\phi^{-1}(t))\right|
=
\left|F(x) - \sum_{k=-M}^N F(kh)S(k,h)(x)\right|.
\]
The idea for giving the error bound is
to divide the error into the following two terms:
\begin{align*}
 &\left|F(x) - \sum_{k=-M}^N F(kh)S(k,h)(x)\right|\\
&\leq\left|F(x) - \sum_{k=-\infty}^{\infty} F(kh)S(k,h)(x)\right|
+\left|\sum_{k=-\infty}^{-M-1}F(kh)S(k,h)(x)
+\sum_{k=N+1}^{\infty}F(kh)S(k,h)(x)
\right|.
\end{align*}
The first and second terms are referred to as the discretization and truncation errors, respectively.
We estimate the discretization error as follows.
The proof is provided in Section~\ref{sec:discretization}.
\begin{lem}
\label{lem:discretization}
Let the assumptions of Theorem~\ref{thm:new} be fulfilled.
Then, putting $F(x)=f(\phi(x))$, we have
\[
\sup_{x\in\mathbb{R}}
 \left|F(x) - \sum_{k=-\infty}^{\infty} F(kh)S(k,h)(x)\right|
\leq \frac{4K}{\pi d \mu(1 - \rme^{-2\pi d/h})\{\cos(d/2)\}^{\alpha+\beta}}
\left(\frac{\rme}{\rme - 1}\right)^{\mu/2} \rme^{-\pi d /h}.
\]
\end{lem}

We estimate the truncation error as follows.
The proof is provided in Section~\ref{sec:truncation}.
\begin{lem}
\label{lem:truncation}
Let the assumptions of Theorem~\ref{thm:new} be fulfilled.
Then, putting $F(x)=f(\phi(x))$, we have
\[
\sup_{x\in\mathbb{R}}
\left|\sum_{k=-\infty}^{-M-1}F(kh)S(k,h)(x)
+\sum_{k=N+1}^{\infty}F(kh)S(k,h)(x)
\right|
\leq \frac{2K}{\mu h} \rme^{-\mu n h}.
\]
\end{lem}

Combining Lemmas~\ref{lem:discretization} and~\ref{lem:truncation}
with $h$ in~\eqref{eq:def-h},
we establish Theorem~\ref{thm:new}, which completes the proof.

\subsection{Estimate of the discretization error}
\label{sec:discretization}

The following function space is required to estimate the discretization error.

\begin{defn}
Let $\domD_d(\epsilon)$ be a rectangular domain defined
for $0<\epsilon<1$ by
\[
\domD_d(\epsilon)
= \{\zeta\in\mathbb{C}:|\Re\zeta|<1/\epsilon,\, |\Im\zeta|<d(1-\epsilon)\}.
\]
Then, $\Hone(\domD_d)$ denotes the family of all functions $F$
that are analytic in $\domD_d$ such that the norm $\mathcal{N}_1(F,d)$ is finite, where
\begin{equation}
\mathcal{N}_1(F,d)
=\lim_{\epsilon\to 0}\oint_{\partial \domD_d(\epsilon)} |F(\zeta)||\divv \zeta|.
\label{eq:Def-N_1}
\end{equation}
\end{defn}

The discretization error for
a function $F$ belonging to $\Hone(\domD_d)$
has been estimated as follows.

\begin{thm}[Stenger~{\cite[Theorem~3.1.3]{stenger93:_numer_method}}]
\label{Thm:Sinc-Infinite-Sum-Approx}
Let $F\in\Hone(\domD_d)$. Then,
\[
    \sup_{x\in\mathbb{R}}
    \left|F(x)-\sum_{k=-\infty}^{\infty}F(kh)S(k,h)(x)\right|
\leq\frac{\mathcal{N}_1(F,d)}{\pi d(1-\rme^{-2\pi d/h})}\rme^{-\pi d/h}.
\]
\end{thm}

According to Theorem~\ref{Thm:Sinc-Infinite-Sum-Approx},
Lemma~\ref{lem:discretization} is derived if the next lemma is shown.
\begin{lem}
\label{lem:bound-None-F}
Let the assumptions of Theorem~\ref{thm:new} be fulfilled.
Then, putting $F(x)=f(\phi(x))$, we have $F\in\Hone(\domD_d)$ and
\begin{equation}
 \mathcal{N}_1(F,d)
\leq \frac{4K}{\mu\cos^{\alpha+\beta}(d/2)}
\left(\frac{\rme}{\rme - 1}\right)^{\mu/2}.
\label{eq:bound-None-F}
\end{equation}
\end{lem}

The following two lemmas are essential to show Lemma~\ref{lem:bound-None-F}.

\begin{lem}[Okayama et al.~{\cite[Lemma 4.21]{okayama13:_error_sinc}}]
\label{lem:exp-bound}
For all $x\in\mathbb{R}$ and $y\in(-\pi,\pi)$, we have
\begin{align*}
\left|\frac{1}{1+\rme^{x+\imnum y}}\right|
&\leq \frac{1}{(1+\rme^x)\cos(y/2)},\\
\left|\frac{1}{1+\rme^{-(x+\imnum y)}}\right|
&\leq \frac{1}{(1+\rme^{-x})\cos(y/2)}.
\end{align*}
\end{lem}
\begin{lem}
\label{lem:essential-bound}
For all $\zeta\in\overline{\domD_{\pi}}$, it holds that
\begin{equation}
 \left|
  \frac{\log(1+\rme^{\zeta})}{1+\log(1+\rme^{\zeta})}
  \cdot \frac{\rme^{-l} + \rme^{\zeta}}{\rme^{\zeta}}
\right|
\leq 1,
\label{eq:essential-ineq}
\end{equation}
where $l = \log(\rme/(\rme - 1))$.
\end{lem}

We prefer to give the proof of Lemma~\ref{lem:essential-bound}
at the end of this section (Section~\ref{sec:essential-ineq})
because it requires preparation and a long discussion.
If we accept this lemma here,
Lemma~\ref{lem:bound-None-F} is shown as follows.

\begin{proof}
Since $f$ is analytic in $\phi(\domD_d)$,
$F$ is analytic in $\domD_d$.
Therefore, the remaining task is to show~\eqref{eq:bound-None-F}.
From the definition~\eqref{eq:Def-N_1},
$\mathcal{N}_1(F,d)$ is expressed as
\begin{align}
\mathcal{N}_1(F,d)
&=\int_{-\infty}^{\infty}|F(x+\imnum d)|\diff x
+ \int_{-\infty}^{\infty}|F(x-\imnum d)|\diff x\nonumber\\
&\quad+\lim_{x\to-\infty}
\int_{-d}^d |F(x+\imnum y)|\diff y
 + \lim_{x\to\infty}
\int_{-d}^d |F(x+\imnum y)|\diff y.
\label{eq:None-F-expression}
\end{align}
First, we estimate $|F(x+\imnum y)|$.
Using~\eqref{eq:f-bound},~\eqref{eq:essential-ineq},
and Lemma~\ref{lem:exp-bound}, we obtain
\begin{align*}
|F(x+\imnum y)|
&\leq K
\left|
 \frac{\log(1+\rme^{x+\imnum y})}{1+\log(1+\rme^{x+\imnum y})}
\right|^{\alpha}
\left|
 \frac{1}{1+\rme^{x+\imnum y}}
\right|^{\beta}\nonumber\\
&\leq
\frac{K}{|1 + \rme^{-(x+l + \imnum y)}|^{\alpha}|1+\rme^{x+\imnum y}|^{\beta}}
\nonumber\\
&\leq
\frac{K}{(1+\rme^{-(x+l)})^{\alpha}(1+\rme^x)^{\beta}\cos^{\alpha+\beta}(y/2)},
\end{align*}
where $l=\log(\rme/(\rme - 1))$.
According to this estimate, the third and fourth terms of~\eqref{eq:None-F-expression}
are bounded as
\begin{align*}
&\lim_{x\to-\infty}
\int_{-d}^d |F(x+\imnum y)|\diff y
 + \lim_{x\to\infty}
\int_{-d}^d |F(x+\imnum y)|\diff y\\
&\leq \lim_{x\to-\infty}
\frac{K}{(1+\rme^{-(x+l)})^{\alpha}(1+\rme^x)^{\beta}}
\int_{-d}^d \frac{\divv y}{\cos^{\alpha+\beta}(y/2)}
+\lim_{x\to\infty}
\frac{K}{(1+\rme^{-(x+l)})^{\alpha}(1+\rme^x)^{\beta}}
\int_{-d}^d \frac{\divv y}{\cos^{\alpha+\beta}(y/2)}\\
&= 0.
\end{align*}
By the same estimate,
the first and second terms of~\eqref{eq:None-F-expression} are bounded as
\begin{align*}
&\int_{-\infty}^{\infty}|F(x+\imnum d)|\diff x
+ \int_{-\infty}^{\infty}|F(x-\imnum d)|\diff x\\
&\leq \frac{K}{\cos^{\alpha+\beta}(d/2)}
\int_{-\infty}^{\infty}
\frac{\divv x}{(1+\rme^{-(x+l)})^{\alpha}(1+\rme^x)^{\beta}}
+\frac{K}{\cos^{\alpha+\beta}(-d/2)}
\int_{-\infty}^{\infty}
\frac{\divv x}{(1+\rme^{-(x+l)})^{\alpha}(1+\rme^x)^{\beta}}\\
&\leq\frac{2K}{\cos^{\alpha+\beta}(d/2)}
\int_{-\infty}^{\infty}
\frac{\divv x}{(1+\rme^{-(x+l)})^{\mu}(1+\rme^x)^{\mu}},
\end{align*}
where $\mu=\min\{\alpha,\beta\}$.
By changing the variable $x$ as $x=t-(l/2)$, we obtain
\begin{align*}
\frac{2K}{\cos^{\alpha+\beta}(d/2)}
\int_{-\infty}^{\infty}
\frac{\divv x}{(1+\rme^{-(x+l)})^{\mu}(1+\rme^x)^{\mu}}
&=\frac{2K}{\cos^{\alpha+\beta}(d/2)}
\int_{-\infty}^{\infty}
\frac{\divv t}{(1+\rme^{-t-\frac{l}{2}})^{\mu}(1+\rme^{t-\frac{l}{2}})^{\mu}}\\
&=\frac{4K}{\cos^{\alpha+\beta}(d/2)}
\int_{0}^{\infty}
\frac{\divv t}{(1+\rme^{-t-\frac{l}{2}})^{\mu}(1+\rme^{t-\frac{l}{2}})^{\mu}}\\
&=\frac{4K}{\cos^{\alpha+\beta}(d/2)}
\int_{0}^{\infty}
\frac{\rme^{-\mu t}}
     {(1+\rme^{-t-\frac{l}{2}})^{\mu}(\rme^{-t}+\rme^{-\frac{l}{2}})^{\mu}}
     \diff t\\
&\leq\frac{4K}{\cos^{\alpha+\beta}(d/2)}
\int_{0}^{\infty}
\frac{\rme^{-\mu t}}
     {(1+0)^{\mu}(0+\rme^{-\frac{l}{2}})^{\mu}}
     \diff t\\
&=\frac{4K}{\cos^{\alpha+\beta}(d/2)}\cdot\frac{(\rme^{l})^{\mu/2}}{\mu},
\end{align*}
which is the desired result.
\end{proof}

\subsection{Estimate of the truncation error}
\label{sec:truncation}

The following lemma is useful to the proof of Lemma~\ref{lem:truncation}.

\begin{lem}
For all $x\in\mathbb{R}$, we have
\begin{equation}
\left|
\frac{\log(1+\rme^x)}{1+\log(1+\rme^x)}\cdot
\frac{1+\rme^x}{\rme^x}
\right|\leq 1.
\label{eq:bound-func-on-real}
\end{equation}
\end{lem}
\begin{proof}
Putting $t=\log(1+\rme^x)$ and noting $t>0$,
we reformulate~\eqref{eq:bound-func-on-real} as
\begin{align*}
\left|
\frac{\log(1+\rme^x)}{1+\log(1+\rme^x)}\cdot
\frac{1+\rme^x}{\rme^x}
\right|
=\frac{t}{1+t}\cdot\frac{\rme^{t}}{\rme^{t} - 1} \leq 1
\quad \Leftrightarrow \quad \rme^t - t - 1\geq 0.
\end{align*}
Since $(\rme^t - t - 1)'=\rme^t - 1\geq 0$,
we have $\rme^t - t - 1\geq \rme^0 - 0 - 1=0$,
which is equivalent to~\eqref{eq:bound-func-on-real}.
\end{proof}

Using the above lemma, we prove Lemma~\ref{lem:truncation} as follows.

\begin{proof}
Using~\eqref{eq:f-bound},~\eqref{eq:bound-func-on-real}
and $|S(k,h)(x)|\leq 1$, we have
\begin{align*}
&\left|\sum_{k=-\infty}^{-M-1}F(kh)S(k,h)(x)
+\sum_{k=N+1}^{\infty}F(kh)S(k,h)(x)
\right|\\
&\leq
\sum_{k=-\infty}^{-M-1}|f(\phi(kh))|
+\sum_{k=N+1}^{\infty}|f(\phi(kh))|\\
&\leq K\sum_{k=-\infty}^{-M-1}
\left|\frac{\log(1+\rme^{kh})}{1+\log(1+\rme^{kh})}\right|^{\alpha}
\frac{1}{|1+\rme^{kh}|^{\beta}}
+K\sum_{k=N+1}^{\infty}
\left|\frac{\log(1+\rme^{kh})}{1+\log(1+\rme^{kh})}\right|^{\alpha}
\frac{1}{|1+\rme^{kh}|^{\beta}}\\
&\leq
K\sum_{k=-\infty}^{-M-1}
\left|\frac{\rme^{kh}}{1+\rme^{kh}}\right|^{\alpha}
\frac{1}{|1+\rme^{kh}|^{\beta}}
+K\sum_{k=N+1}^{\infty}
\left|\frac{\rme^{kh}}{1+\rme^{kh}}\right|^{\alpha}
\frac{1}{|1+\rme^{kh}|^{\beta}}.
\end{align*}
Furthermore, the second term is estimated as
\begin{align*}
K\sum_{k=N+1}^{\infty}
\left|\frac{\rme^{kh}}{1+\rme^{kh}}\right|^{\alpha}
\frac{1}{|1+\rme^{kh}|^{\beta}}
=\sum_{k=N+1}^{\infty}
\frac{K\rme^{-\beta kh}}{(1+\rme^{-kh})^{\alpha+\beta}}
\leq \sum_{k=N+1}^{\infty}
\frac{K\rme^{-\beta kh}}{(1 + 0)^{\alpha+\beta}}
\leq \frac{K}{h}\int_{Nh}^{\infty}\rme^{-\beta x}\diff x
=\frac{K}{\beta h}\rme^{-\beta N h}.
\end{align*}
In the same manner, the first term is estimated as
\[
K\sum_{k=-\infty}^{-M-1}
\left|\frac{\rme^{kh}}{1+\rme^{kh}}\right|^{\alpha}
\frac{1}{|1+\rme^{kh}|^{\beta}}
\leq \frac{K}{\alpha h}\rme^{-\alpha M h}.
\]
Using the relation~\eqref{eq:def-MN} and $\mu=\min\{\alpha,\beta\}$,
we have
\[
\frac{K}{\alpha h}\rme^{-\alpha M h}
+\frac{K}{\beta h}\rme^{-\beta N h}
\leq
\frac{K}{\alpha h}\rme^{-\mu n h}
+\frac{K}{\beta h}\rme^{-\mu n h}
\leq
\frac{K}{\mu h}\rme^{-\mu n h}
+\frac{K}{\mu h}\rme^{-\mu n h},
\]
which shows the desired inequality.
\end{proof}

\subsection{Proof of Lemma~\ref{lem:essential-bound}}
\label{sec:essential-ineq}

Our project is completed by showing Lemma~\ref{lem:essential-bound}.
For this purpose,
we prepare the following lemma.

\begin{lem}
\label{lem:final-prepare}
For all $t\leq 0$, it holds that
\[
 1 + \rme^{t} (\rme^{t+1} -1  + t + t^2 - \rme(1 + t + t^2))\geq 0.
\]
\end{lem}
\begin{proof}
Put $p(t)$ and $q(t)$ as
\begin{align*}
 p(t)&=1 + \rme^{t} (\rme^{t+1} -1  + t + t^2 - \rme(1 + t + t^2)),\\
 q(t)&=2\rme^{t+1} - \rme(1+t)(2+t) + t(3+t).
\end{align*}
Then, the derivative of $p$ is expressed as
$p'(t) = \rme^{t} q(t)$.
Since the signs of $p'(t)$ and $q(t)$ are the same,
we investigate $q(t)$. Differentiating $q(t)$ as
\begin{align*}
 q'(t) &= 2 \rme^{t+1} + (1-\rme)(2t + 3),\\
 q''(t)&= 2(\rme^{t+1} + 1 - \rme),
\end{align*}
we have $q''(\log((\rme - 1)/\rme))=0$.
From $q'(-1)=q'(0)=3-\rme>0$ and $q'(\log((\rme - 1)/\rme))<0$,
there exist unique points $a$ and $b$
such that $q'(a)=q'(b)=0$
with $-1<a<\log((\rme - 1)/\rme)$
and $\log((\rme - 1)/\rme)<b<0$.
From the monotonicity property on the intervals $-1<t<a$
and $b<t<0$,
we have $q(a)>q(-1)=0$ and $q(b)<q(0)=0$,
which derives $q(a)>0>q(b)$.
Therefore,
there exists a unique point $c$
such that $q(c)=0$ with $a<c<b$.
Using this $c$, we have
$q(t)< 0$ for $t<-1$,
$q(-1)=0$,
$q(t)>0$ for $-1<t<c$,
$q(c)=0$,
$q(t)<0$ for $c<t<0$, and
$q(0)=0$.
By summing up the above arguments, we can determine
the sign of $p'(t)=\rme^t q(t)$, from which we can conclude that
$p(t)$ has its minimum at $t=-1$ and $t=0$, i.e.,
$p(t)\geq p(-1)=p(0)=0$.
\end{proof}

We are in a position to prove Lemma~\ref{lem:essential-bound}.

\begin{proof}
By the maximum modulus principle,
the left-hand side of~\eqref{eq:essential-ineq}
has its maximum on the boundary of $\overline{\domD_{\pi}}$,
i.e., $\zeta=x\pm\imnum \pi$.
From the symmetry with respect to the real axis,
it is sufficient to consider $\zeta=x+\imnum \pi$.
Therefore, we show
\begin{equation}
 f(x):= \left|
\frac{\log(1+\rme^{x+\imnum \pi})}{1+\log(1+\rme^{x+\imnum \pi})}
\cdot
\frac{\rme^{-l}+\rme^{x+\imnum \pi}}{\rme^{x + \imnum \pi}}
\right|\leq 1.
\label{eq:goal-ineq}
\end{equation}
In the case where $x=0$, we have
\[
 f(0) = \lim_{x\to 0} f(x) = \left|\frac{1}{\rme}\right| < 1.
\]
The remaining cases are (i) $x>0$ and (ii) $x<0$, which we consider independently.
Note that
\[
 \log(1+\rme^{x+\imnum \pi}) = \log(1-\rme^x)
= \log|1 - \rme^x| + \imnum \arg(1 - \rme^x)
=
\begin{cases}
\log(\rme^x - 1) + \imnum \pi & (x > 0),\\
\log(1 - \rme^x) & (x < 0).
\end{cases}
\]
(i) $x>0$. In this case, $f(x)$ is expressed as
\[
 f(x) =
\left|
 \frac{\log(\rme^x - 1) + \imnum \pi}{1+\log(\rme^x - 1) + \imnum \pi}
\right|
\left|
 \frac{\rme^{-l}-\rme^x}{-\rme^x}
\right|.
\]
By putting $t=\log(\rme^x - 1)$,
$f(x)\leq 1$ for $x>0$ is equivalent to
\[
 g(t):=
\left|\frac{t+\imnum \pi}{1 + t + \imnum \pi}\right|
\left|\frac{\rme^{-l}-\rme^{t}-1}{-\rme^{t}-1}\right|
\leq 1
\]
for $t\in\mathbb{R}$. This is proved by showing
\begin{equation}
 \{g(t)\}^2 = \frac{t^2 + \pi^2}{(1+t)^2 + \pi^2}
\left(1 - \frac{\rme - 1}{\rme (\rme^t + 1)}\right)^2
\leq 1
\label{eq:g-square-ineq}
\end{equation}
for $t\in\mathbb{R}$.
In the following, we show~\eqref{eq:g-square-ineq}
for two cases: (i-a) $t>-1/2$ and (i-b) $t\leq -1/2$.

\begin{enumerate}
 \item[(i-a)] $t>-1/2$. We begin with an obvious inequality:
\[
 0\leq (\rme^{1/2} + 1)^2 = \rme + 2 \rme^{1/2} + 1,
\]
which is equivalent to
$\rme - 1 \leq 2\rme + 2\rme^{1/2}$, and further equivalent to
\[
 \frac{1}{1 + \rme^{-1/2}}\leq \frac{2\rme}{\rme - 1}.
\]
From this and $t > -1/2$, we have
\[
 0\leq \frac{1}{\rme^t + 1}\leq \frac{1}{1+\rme^{-1/2}}
\leq \frac{2\rme}{\rme - 1},
\]
from which we have
\[
 -2\leq -\frac{\rme - 1}{\rme(\rme^t + 1)}\leq 0
\quad \Leftrightarrow \quad
\left(1 - \frac{\rme - 1}{\rme(\rme^t + 1)}\right)^2\leq 1.
\]
Therefore, it holds for $t>-1/2$ that
\[
 \{g(t)\}^2\leq \frac{t^2 + \pi^2}{(0 + t)^2 + \pi^2}
\left(1 - \frac{\rme - 1}{\rme(\rme^t + 1)}\right)^2
=\left(1 - \frac{\rme - 1}{\rme(\rme^t + 1)}\right)^2
\leq 1.
\]
 \item[(i-b)] $t\leq -1/2$. First, for all $t\in\mathbb{R}$,
it holds that
\[
 (\rme - 1)\left(t + \frac{\rme}{\rme - 1}\right)^2
+\frac{\pi^2(\rme - 1)^2 - \rme}{\rme-1} \geq 0,
\]
which is equivalent to
\[
 t^2 + \pi^2 \leq \rme((1+t)^2 + \pi^2)
\quad \Leftrightarrow \quad
\frac{t^2+\pi^2}{\rme((1+t)^2 + \pi^2)} \leq 1.
\]
Therefore, it holds for $t\leq -1/2$ that
\[
\{g(t)\}^2
\leq \frac{t^2 + \pi^2}{(1+t)^2 + \pi^2}
\left(1 - \frac{\rme - 1}{\rme(\rme^{-1/2}+ 1)}\right)^2
=\frac{t^2 + \pi^2}{\rme((1+t)^2 + \pi^2)} \leq 1.
\]
\end{enumerate}
(ii) $x<0$. In this case, $f(x)$ is expressed as
\[
 f(x) = \left|
 \frac{\log(1-\rme^x)}{1+\log(1-\rme^x)}
 \left(1 - \frac{\rme-1}{\rme}\rme^{-x}\right)
\right|.
\]
By putting $t=\log(1 - \rme^x)$,
$f(x)\leq 1$ for $x < 0$ is equivalent to
\[
 h(t):=\left|
\frac{t}{1+t}\left(1 - \frac{\rme - 1}{\rme(1-\rme^t)}\right)
\right|\leq 1
\]
for $t\leq 0$, which is shown as follows. In the case where $t=0$ and $t=-1$,
by L'H\^{o}pital's rule, we have
\begin{align*}
h(0)&=\lim_{t\to 0}\left|
\frac{t\rme^{t+1} - t}{\rme^{t+1}-\rme + t\rme^{t+1} - t\rme}
\right|
=\lim_{t\to 0}\left|
\frac{\rme^{t+1} + t\rme^{t+1} - t}{2\rme^{t+1} + t \rme^{t+1} - \rme}
\right|
=\frac{\rme - 1}{\rme} < 1,\\
h(-1)&=\lim_{t\to -1}\left|
\frac{t\rme^{t+1} - t}{\rme^{t+1}-\rme + t\rme^{t+1} - t\rme}
\right|
=\lim_{t\to -1}\left|
\frac{\rme^{t+1} + t\rme^{t+1} - t}{2\rme^{t+1} + t \rme^{t+1} - \rme}
\right|
=\frac{1}{\rme - 1} < 1.
\end{align*}
The remaining cases are $t<-1$ and $-1<t<0$.
In consideration of the signs of $t/(1+t)$ and
$1 - (\rme - 1)/(\rme(1 - \rme^t))$,
we can remove the absolute value sign from $h(t)$ as
\[
 h(t) = \frac{t}{1+t}\left(1 - \frac{\rme - 1}{\rme(1-\rme^t)}\right).
\]
According to Lemma~\ref{lem:final-prepare}, it holds for all $t\leq 0$ that
\[
 h'(t) =
\frac{1 + \rme^t(\rme^{t+1} - 1 + t + t^2 - \rme(1+t+t^2))}
     {\rme(\rme^t -1)^2(1+t)^2}
\geq 0.
\]
In other words, $h(t)$ increases monotonically for $t\leq 0$.
Thus, we obtain $h(t)\leq h(-1) < 1$ for $t<-1$ and
$h(t)\leq h(0) < 1$ for $-1 <t< 0$.
This completes the proof.
\end{proof}

%% file: conclusion.tex
\section{Concluding remarks}
\label{sec:conclusion}

The Sinc approximation is an approximation formula
on the infinite interval $(-\infty,\infty)$; thus,
an appropriate conformal map is required to apply the Sinc approximation
on other intervals.
In the case of exponentially decaying functions
that are defined on the semi-infinite interval $(0,\infty)$,
Stenger~\cite{stenger93:_numer_method} proposed the use of $t=\psi(x)$
to map $(-\infty,\infty)$ onto $(0,\infty)$ and derived the
approximation formula~\eqref{eq:Stenger-approx}.
This paper has proposed the use of $t=\phi(x)$ rather than $t=\psi(x)$
and derived a new approximation formula~\eqref{eq:Shintaku-approx}.
Through a theoretical analysis, we have given
a computable error bound for the proposed approximation formula,
which is quite useful for construction
of algorithms with automatic result verification
in arbitrary-precision arithmetic.
By comparing these two approximation formulas
theoretically and numerically,
we have demonstrated the superiority of the proposed formula.

This improvement can be extended to many other numerical methods based on
the Sinc approximation combined with the conformal map $t=\psi(x)$.
For example, numerical methods for
the Laplace transform~\cite{stenger93:_numer_method},
the Laplace transform inversion~\cite{stenger95:_colloc},
initial value problems~\cite{stenger99:_ode_ivp},
second-order differential equations~\cite{stenger11:_handb}, and
Wiener--Hopf equations~\cite{stenger00:_summar}.
Replacing the conformal map $t=\psi(x)$ with $t=\phi(x)$
in such methods may achieve faster convergence.
In future, we plan to conduct theoretical analyses of such cases.

%% file: SemiInfSE-Sinc.bbl
\begin{thebibliography}{10}
\expandafter\ifx\csname url\endcsname\relax
  \def\url#1{\texttt{#1}}\fi
\expandafter\ifx\csname urlprefix\endcsname\relax\def\urlprefix{URL }\fi
\expandafter\ifx\csname href\endcsname\relax
  \def\href#1#2{#2} \def\path#1{#1}\fi

\bibitem{stenger93:_numer_method}
F.~Stenger, Numerical Methods Based on Sinc and Analytic Functions,
  Springer-Verlag, New York, 1993.

\bibitem{lund92:_sinc}
J.~Lund, K.~L. Bowers, Sinc Methods for Quadrature and Differential Equations,
  SIAM, Philadelphia, PA, 1992.

\bibitem{stenger00:_summar}
F.~Stenger, Summary of {Sinc} numerical methods, Journal of Computational and
  Applied Mathematics 121 (2000) 379--420.

\bibitem{sugihara04:_recen}
M.~Sugihara, T.~Matsuo, Recent developments of the {Sinc} numerical methods,
  Journal of Computational and Applied Mathematics 164--165 (2004) 673--689.

\bibitem{stenger11:_handb}
F.~Stenger, Handbook of Sinc Numerical Methods, CRC Press, Boca Raton, FL,
  2011.

\bibitem{Machida}
T.~Okayama, K.~Machida, Error estimate with explicit constants for the
  trapezoidal formula combined with {Muhammad--Mori's} {SE} transformation for
  the semi-infinite interval, JSIAM Letters 9 (2017) 45--47.

\bibitem{Hara}
R.~Hara, T.~Okayama, Explicit error bound for {Muhammad--Mori's} {SE-Sinc}
  indefinite integration formula over the semi-infinite interval, Proceedings
  of the 2017 International Symposium on Nonlinear Theory and its Applications
  (2017) 677--680.

\bibitem{okayama16:_error}
T.~Okayama, Error estimates with explicit constants for the {Sinc}
  approximation over infinite intervals, Applied Mathematics and Computation
  319 (2018) 125--137.

\bibitem{okayama13:_error_sinc}
T.~Okayama, T.~Matsuo, M.~Sugihara, Error estimates with explicit constants for
  {Sinc} approximation, {Sinc} quadrature and {Sinc} indefinite integration,
  Numerische Mathematik 124 (2013) 361--394.

\bibitem{stenger95:_colloc}
F.~Stenger, Collocating convolutions, Mathematics of Computation 64 (1995)
  211--235.

\bibitem{stenger99:_ode_ivp}
F.~Stenger, S.~Gustafson, B.~Keyes, M.~O'Reilly, K.~Parker, {ODE-IVP-PACK} via
  {Sinc} indefinite integration and {Newton's} method, Numerical Algorithms 20
  (1999) 241--268.

\end{thebibliography}
